\documentclass{amsart}

\usepackage{amsmath}
\usepackage{amsfonts}
\usepackage{amssymb,enumerate}
\usepackage{amsthm}
\usepackage[all]{xy}
\usepackage{hyperref}

\theoremstyle{plain}
\newtheorem{lem}{Lemma}[section]

\newtheorem{prop}[lem]{Proposition}
\newtheorem{thm}[lem]{Theorem}

\theoremstyle{definition}
\newtheorem{defn}[lem]{Definition}

\newtheorem{disc}[lem]{Remark}

\newtheorem{fact}[lem]{Fact}

\newtheorem{notation}[lem]{Notation}

\newcommand{\cat}[1]{\mathcal{#1}}

\newcommand{\catg}{\cat{G}}
\newcommand{\catp}{\cat{P}}

\newcommand{\cati}{\cat{I}}
\newcommand{\cata}{\cat{A}}

\newcommand{\catb}{\cat{B}}
\newcommand{\catgi}{\cat{GI}}

\newcommand{\catgic}{\cat{GI}_C}

\newcommand{\catgc}{\cat{G}_C}

\newcommand{\catac}{\cat{A}_C}

\newcommand{\catbc}{\cat{B}_C}

\newcommand{\catic}{\cat{I}_C}

\newcommand{\catpc}{\cat{P}_C}

\newcommand{\opg}{\cat{G}}

\newcommand{\pdim}{\operatorname{pd}}	
\newcommand{\pd}{\operatorname{pd}}	
\newcommand{\gdim}{\mathrm{G}\text{-}\!\dim}	
\newcommand{\gkdim}[1]{\mathrm{G}_{#1}\text{-}\!\dim}	
\newcommand{\gcdim}{\gkdim{C}}	
	
\newcommand{\id}{\operatorname{id}}

\newcommand{\gid}{\catid{G}}

\newcommand{\pcpd}{\catpc\text{-}\pd}

\newcommand{\icid}{\catic\text{-}\id}

\newcommand{\pcdim}{\catpc\text{-}\pd}

\newcommand{\ext}{\operatorname{Ext}}

\newcommand{\HH}{\operatorname{H}}
\newcommand{\Hom}{\operatorname{Hom}}	

\newcommand{\spec}{\operatorname{Spec}}

\newcommand{\tor}{\operatorname{Tor}}
\newcommand{\im}{\operatorname{Im}}

\newcommand{\Ker}{\operatorname{Ker}}

\newcommand{\ideal}[1]{\mathfrak{#1}}
\newcommand{\m}{\ideal{m}}

\newcommand{\p}{\ideal{p}}

\newcommand{\ass}{\operatorname{Ass}}

\newcommand{\supp}{\operatorname{Supp}}

\newcommand{\xra}{\xrightarrow}

\renewcommand{\geq}{\geqslant}
\renewcommand{\leq}{\leqslant}

\renewcommand{\hom}{\Hom}

\newcommand{\Ext}[4][R]{\operatorname{Ext}_{#1}^{#2}(#3,#4)}

\newcommand{\Otimes}[3][R]{#2\otimes_{#1}#3}
\renewcommand{\Hom}[3][R]{\operatorname{Hom}_{#1}(#2,#3)}	
\newcommand{\Tor}[4][R]{\operatorname{Tor}^{#1}_{#2}(#3,#4)}

\newcommand{\cd}{C^{\dagger}}
\newcommand{\gc}{\text{G}_C}

\theoremstyle{definition}
\newtheorem{proofn}[lem]{Proof}

\numberwithin{equation}{lem}

\begin{document}

\bibliographystyle{amsplain}

\author{Sean Sather-Wagstaff}

\address{Mathematics Department,
NDSU Dept \# 2750,
PO Box 6050,
Fargo, ND 58108-6050
USA}

\email{Sean.Sather-Wagstaff@ndsu.edu}

\urladdr{http://www.ndsu.edu/pubweb/\~{}ssatherw/}

\thanks{This material is based on work supported by North Dakota EPSCoR and 
National Science Foundation Grant EPS-0814442.
The author was supported in part by a grant from the NSA}

\title{Embedding modules of finite homological dimension}

\date{\today}


\begin{abstract}
This paper builds on work of Hochster and Yao that provides nice embeddings for finitely generated modules of
finite G-dimension, finite projective dimension, or locally finite injective dimension.
We extend these results by providing similar embeddings in the relative setting,
that is, for certain modules of finite
$\gc$-dimension, finite $\catpc$-projective dimension, locally finite $\catgic$-injective dimension,
or locally finite $\catic$-injective dimension where $C$ is a semidualizing module. 
Along the way, we extend some results for modules of finite homological dimension 
to modules of locally finite homological dimension in the relative setting.
\end{abstract}

\maketitle

\section{Introduction} \label{sec0}

Throughout this note, $R$ is a commutative noetherian ring with identity.
The purpose of this note is to 
extend some results of Hochster and Yao, beginning with the following,
wherein $\underline z_{[i]}=z_1,\ldots,z_i$ and $\gdim_R(M)$ is the G-dimension of Auslander and Bridger~\cite{auslander:adgeteac,auslander:smt}.\footnote{This
is also now called ``Gorenstein projective dimension'' after the work of Enochs, Jenda, and others. However,  this is a generalization of 
Auslander and Bridger's original notion for non-finitely-generated modules. Since we are focused on finitely generated modules here, we continue with Auslander and Bridger's terminology.}

\begin{thm}[\protect{\cite[Theorem 2.3]{hochster:etmf}}]
\label{thm110527a}
Let  $M$ be a non-zero finitely generated $R$-module.
Then
\begin{enumerate}[\quad\rm(1)]
\item \label{thm110527a1}
If $\gdim_R(M) = r < \infty$, then there are an $R$-regular sequence $\underline z = z_1, \ldots, z_r$,
integers $n_0, n_1, \ldots , n_r  \geq 0$ with $n_r\geq 1$, and an exact sequence
$$0 \to M \to Z \to N \to 0$$
with $Z = \oplus_{i=0}^r(R/(\underline z_{[i]}))^{n_i}$ and $\gdim_R(N) \leq r$.
\item \label{thm110527a2}
If $\pdim_R(M) = r < \infty$, then there exist an $R$-regular sequence $\underline z = z_1, \ldots, z_r$,
integers $n_0, n_1, \ldots , n_r  \geq 0$ with $n_r\geq 1$, and an exact sequence
$$0 \to M \to Z \to N \to 0$$
with $Z = \oplus_{i=0}^r(R/(\underline z_{[i]}))^{n_i}$ and $\pdim_R(N) \leq r$.
\end{enumerate}
\end{thm}

The point of Theorem~\ref{thm110527a}, as Hochster and Yao indicate, is that the given module of finite G-dimension or projective dimension
embeds in a module that \emph{obviously} has finite projective dimension, with some extra control on the G-dimension or projective dimension
of the cokernel.

The first of our results, stated next, is for modules of finite $\gc$-dimension and
modules of finite $\catpc$-projective dimension, where $C$ is a
semidualizing $R$-module, after Golod~\cite{golod:gdagpi}, Holm and J\o rgensen~\cite{holm:smarghd}
and White~\cite{white:gpdrsm}.\footnote{It is worth noting that
these notions appeared (more or less) implicitly in several places. However, to the best of our knowledge these are the first places
where $\gc$-dimension and $\catpc$-projective dimension were studied explicitly. }
(See Section~\ref{sec110530a} for definitions and background material.)
On the one hand, this  result is a
corollary to Theorem~\ref{thm110527a}.
On the other hand, Theorem~\ref{thm110527a} is the special case $C=R$, so our result generalizes this one.
This result is proved in Proofs~\ref{proof110530a} and~\ref{proof110530b},
which are similar to the proof of \cite[Theorem 4.2]{hochster:etmf}.

\begin{thm}
\label{thm110527b}
Let  $M$ be a non-zero finitely generated $R$-module, and let $C$ be a semidualizing $R$-module.
Then
\begin{enumerate}[\quad\rm(1)]
\item \label{thm110527b1}
If $\gcdim_R(M) = r < \infty$ and $M$ is in the Bass class $\catbc(R)$, then there exist an $R$-regular sequence $\underline z = z_1, \ldots, z_r$,
integers $n_0, n_1, \ldots , n_r  \geq 0$ with $n_r\geq 1$, and an exact sequence
$$0 \to M \to Z \to N \to 0$$
with $Z = \oplus_{i=0}^r(C/(\underline z_{[i]})C)^{n_i}$ and $\gcdim_R(N) \leq r$ and $N\in\catbc(R)$.
\item \label{thm110527b2}
If $\pcdim_R(M) = r < \infty$, then there exist an $R$-regular sequence $\underline z = z_1, \ldots, z_r$,
integers $n_0, n_1, \ldots , n_r  \geq 0$ with $n_r\geq 1$, and an exact sequence
$$0 \to M \to Z \to N \to 0$$
with $Z = \oplus_{i=0}^r(C/(\underline z_{[i]})C)^{n_i}$ and $\pcdim_R(N) \leq r$.
\end{enumerate}
\end{thm}

The next result we extend is the following:

\begin{thm}[\protect{\cite[Theorem 4.2]{hochster:etmf}}]
\label{thm110530a}
Let $R$ be a  Cohen-Macaulay ring with a pointwise dualizing
module $\omega$. Then, for any non-zero finitely generated $R$-module $M$ with locally finite injective
dimension, there exist an integer $r = \pd_R(\Hom\omega M)$, an
$R$-regular sequence $\underline z = z_1, \ldots , z_r$, non-negative integers $n_0, n_1, \ldots , n_r$ with $n_r\geq 1$, and an
exact sequence
$$0 \to M \to Z \to N \to 0$$
with $Z = \oplus_{i=0}^r(\omega/(\underline z_{[i]})\omega)^{n_i}$, 
such that $N$ has locally finite injective dimension and $\pd_R(\Hom\omega N)\leq r$. 
\end{thm}

We extend this in two directions in Theorems~\ref{thm110608a} and~\ref{thm110530b}. First, we have the version for G-injective dimension.
Second, we give versions relative to a semidualizing module.
These are proved in Proofs~\ref{proof110609a}, \ref{proof110609b}, and~\ref{proof110609c}.

\begin{thm}
\label{thm110608a}
Let $R$ be a  Cohen-Macaulay ring with a pointwise dualizing
module $\omega$. For any non-zero finitely generated $R$-module $M$ with locally finite G-injective
dimension, there exist an integer $r = \gdim_R(\Hom\omega M)$, an 
$R$-regular sequence $\underline z = z_1, \ldots , z_r$, non-negative integers $n_0, n_1, \ldots , n_r$ with $n_r\geq 1$, and an
exact sequence
$$0 \to M \to Z \to N \to 0$$
with $Z = \oplus_{i=0}^r(\omega/(\underline z_{[i]})\omega)^{n_i}$
such that $N$ has locally finite G-injective
dimension and $\gdim_R(\Hom\omega N)\leq r$. 
\end{thm}

\begin{thm}
\label{thm110530b}
Let $R$ be a  Cohen-Macaulay ring with a pointwise dualizing
module $\omega$. 
Let  $M$ a non-zero finitely generated $R$-module, and let $C$ be a semidualizing $R$-module.
Set $\cd=\Hom C\omega$. Then
\begin{enumerate}[\quad\rm(1)]
\item \label{thm110530b1}
If $M$ has locally finite $\catgic\text{-}\id$ and $M$ is in the Auslander class $\catac(R)$, then 
there exist an integer $r = \gdim(\Hom \omega{\Otimes CM})$, an
$R$-regular sequence $\underline z = z_1, \ldots , z_r$, non-negative integers $n_0, n_1, \ldots , n_r$ with $n_r\geq 1$, and an
exact sequence
$$0 \to M \to Z \to N \to 0$$
such that $N$ has locally finite $\catgic\text{-}\id$,
and $Z = \oplus_{i=0}^r(\cd/(\underline z_{[i]})\cd)^{n_i}$, $N \in \catac(R)$ and $\gdim(\Hom\omega{\Otimes CN})\leq r$.
\item \label{thm110530b2}
If $M$ has locally finite $\catic\text{-}\id$, then 
there exist an integer 
$$r = \pd_R(\Hom \omega{\Otimes CM})$$
an $R$-regular sequence $\underline z = z_1, \ldots , z_r$, non-negative integers $n_0, n_1, \ldots , n_r$ with $n_r\geq 1$, and an
exact sequence
$$0 \to M \to Z \to N \to 0$$
such that $N$ has locally finite $\catic\text{-}\id$,
and $Z = \oplus_{i=0}^r(\cd/(\underline z_{[i]})\cd)^{n_i}$  and $\pd_R(\Hom\omega{\Otimes CN})\leq r$.
\end{enumerate}
\end{thm}

We summarize the organization of this paper. 
Section~\ref{sec110530a} contains the proof of Theorem~\ref{thm110527b},
along with the necessary background material. 
Section~\ref{sec110608a} similarly treats Theorems~\ref{thm110608a} and~\ref{thm110530b};
this section also contains several lemmas about modules of locally finite G-injective dimension,
locally finite $\gc$-injective dimension, and locally finite $\catic$-injective dimension where $C$ is a semidualizing module,
as these notions have not been developed in the literature as best we know.\footnote{Ths can be
partially explained by the fact that, at the time of the writing of this article,
the localization question  for G-injective dimension is still not completely answered. Remark~\ref{disc110726a}.} 
Note that we do not develop the background material in the most optimal manner,
focusing instead on accessibility. To be specific, we present background material as it is needed 
in Sections~\ref{sec110530a} and~\ref{sec110608a} instead of putting it all in its own section.
Also, we feel that the definition of locally finite $\catgic\text{-}\id$ is a bit much
to swallow on the first bite, so we first discuss the special case of $\gid$ and work our way up.

\section{Proof of Theorem~\ref{thm110527b}} \label{sec110530a}

\begin{notation} \label{notn1xx}
We let $\pd_R(-)$ and $\id_R(-)$ denote
the classical projective dimension
and the classical injective dimension,
respectively.
\end{notation}

The study of semidualizing modules was initiated independently (with different names)
by Foxby~\cite{foxby:gmarm}, Golod~\cite{golod:gdagpi},
and Vasconcelos~\cite{vasconcelos:dtmc}.
For example, a finitely generated projective  $R$-module of rank 1 is semidualizing.

\begin{defn} \label{notation08a}
A finitely generated $R$-module $C$ is \emph{semidualizing} if 
the natural  map $R\to\Hom CC$ is an isomorphism, and
$\ext_R^{\geq 1}(C,C)=0$.
An $R$-module $C$  is  \emph{pointwise dualizing}
if it is semidualizing and  $\id_{R_{\m}}(C_{\m})<\infty$ for all maximal ideals $\m\subset R$.\footnote{In~\cite{hochster:etmf}
the term ``global canonical module'' is used in place of ``pointwise dualizing module''. Our choice follows the terminology of Grothendieck and Hartshorne~\cite{hartshorne:rad}.}
An $R$-module $C$  is  \emph{dualizing}
if it is semidualizing and $\id_R(C)<\infty$.
\end{defn}

\begin{fact}\label{fact110527a}
\

(a) It is straightforward to show that the semidualizing property is local.
That is, if $C$ is a semidualizing $R$-module, then for each multiplicatively closed subset $U\subseteq R$
the localization $U^{-1}C$ is a semidualizing $U^{-1}R$-module; and conversely,
if $M$ is  a finitely generated $R$-module such that $M_{\m}$ is a semidualizing $R_{\m}$-module for each maximal ideal $\m\subset R$,
then $M$ is a semidualizing $R$-module.
Similarly, if $C$ is a (pointwise) dualizing module for $R$, then for each multiplicatively closed subset $U\subseteq R$
the localization $U^{-1}C$ is a (pointwise) dualizing  module for $U^{-1}R$; and conversely,
if $M$ is  a finitely generated $R$-module such that $M_{\m}$ is a dualizing module for $R_{\m}$ for each maximal ideal $\m\subset R$,
then $M$ is a pointwise dualizing module for $R$.
Moreover, an $R$-module $M$ is a dualizing module for $R$ if and only if it is a pointwise dualizing module for $R$
and the Krull dimension of $R$ is finite.

(b) Let $C$ be a semidualizing $R$-module.
The isomorphism $R\cong\Hom CC$ implies that $\supp_R(C)=\spec(R)$ and
$\ass_R(C)=\ass(R)$. Thus, an element $x\in R$ is $C$-regular if and only if it is $R$-regular.
When $x$ is $R$-regular, the quotient $C/xC$ is a semidualizing $R/xR$-module.
Thus, by induction, a sequence $\underline z=z_1,\ldots,z_r\in R$ is $R$-regular if and only if it is $C$-regular;
see~\cite[Theorem 4.5]{frankild:rrhffd}.
Also, from~\cite[Proposition 3.1]{holm:fear} we know that, for all $R$-modules $M\neq 0$,
we have $\Otimes CM\neq 0\neq\Hom CM$.
\end{fact}

The next categories were introduced by 
Foxby~\cite{foxby:gdcmr}
when $C$ is dualizing, and 
by Vasconcelos~\cite[\S 4.4]{vasconcelos:dtmc} for arbitrary $C$, with different notation.

\begin{defn} \label{notation08d}
Let $C$ be a
semidualizing $R$-module.  
The \emph{Auslander class} of $C$ is the class $\catac(R)$
of $R$-modules $M$ such that 
$\tor^R_{\geq 1}(C,M)=0=\ext_R^{\geq 1}(C,C\otimes_R M)$, and
the natural map $M\to\Hom C{C\otimes_R M}$ is an isomorphism.
The \emph{Bass class} of $C$ is the class $\catbc(R)$
of $R$-modules $N$ such that 
$\ext_R^{\geq 1}(C,M)=0=\tor^R_{\geq 1}(C,\Hom CM)$, and 
the natural evaluation map $C\otimes_R\Hom CN\to N$ is an isomorphism.
\end{defn}

\begin{disc}\label{disc110530a}
It is straightforward to show that the Auslander and Bass classes for $C=R$ are trivial:
$\cata_{R}(R)$ and $\catb_R(R)$ both contain all $R$-modules.
\end{disc}

The following notion was  introduced and studied  by
Holm and J\o rgensen~\cite{holm:smarghd}
and White~\cite{white:gpdrsm}.
Note that in the special case $C=R$, we recover the class of projective $R$-modules and the classical projective dimension.

\begin{defn} \label{notation08bxy}
Let $C$ be a semidualizing $R$-module.
Let
$\catpc(R)$ denote the class of modules $M\cong P\otimes_R C$ with $P$  projective.
Modules in $\catpc(R)$  are called \emph{$C$-projective}.
We let $\catpc\text{-}\pd_R(-)$
denote the homological dimension obtained from resolutions  in $\catpc(R)$,
with the convention that $\pcpd_R(0)=-\infty$.\footnote{We observe the same convention for any homological 
dimension of the zero module.}
\end{defn}

\begin{fact} \label{disc01xyz}
Let  $C$ be 
a semidualizing $R$-module.  
The Bass class $\catbc(R)$ contains
all $R$-modules  of finite $\pcpd$; 
see~\cite[Lemmas 4.1 and 5.1, and Corollary 6.3]{holm:fear}.
Given an $R$-module $M$, one has
$\pcpd_R(M)=\pd_R(\Hom CM)$
and 
$\pd_R(M)=\pcpd_R(\Otimes CM)$
by~\cite[Theorem 2.11]{takahashi:hasm}.
In particular, one has
$\pcpd_R(M)<\infty$ if and only if
$\pd_R(\Hom CM)<\infty$,
and one has
$\pd_R(M)<\infty$ if and only if
$\pcpd_R(\Otimes CM)<\infty$.
\end{fact}

\begin{proofn}[Proof of Theorem~\ref{thm110527b}\eqref{thm110527b2}]
\label{proof110530a}
Assume that $M$ is a non-zero finitely generated $R$-module such that
$\pcpd_R(M) = r < \infty$.
Then Fact~\ref{disc01xyz} implies that
$\pd_R(\Hom CM)=r<\infty$.
Since $M\neq 0$,  we have $\Hom CM\neq 0$, by Fact~\ref{fact110527a}(b), so
Theorem~\ref{thm110527a}\eqref{thm110527a2}
provides an $R$-regular sequence $\underline z = z_1, \ldots, z_r$,
integers $n_0, n_1, \ldots , n_r  \geq 0$ with $n_r\geq 1$, and an exact sequence
\begin{equation}\label{eq110530a}
0 \to \Hom CM \to Z' \to N' \to 0
\end{equation}
with $Z' = \oplus_{i=0}^r(R/(\underline z_{[i]}))^{n_i}$ and $\pd_R(N') \leq r$.
In particular, we have $N'\in\catac(R)$ which implies that $\Tor 1C{N'}=0$.
Thus, an application of $\Otimes C-$ to the sequence~\eqref{eq110530a}
yields the next exact sequence:
\begin{equation}\label{eq110530b}
0 \to \Otimes{C}{\Hom CM} \to \Otimes{C}{Z'} \to \Otimes{C}{N'} \to 0.
\end{equation}
Fact~\ref{disc01xyz} implies that $M\in\catbc(R)$,
so we have $\Otimes{C}{\Hom CM}\cong M$. 
The equality $Z' = \oplus_{i=0}^r(R/(\underline z_{[i]}))^{n_i}$ implies that
$Z:=\Otimes{C}{Z'} = \oplus_{i=0}^r(C/(\underline z_{[i]})C)^{n_i}$.
Because of Fact~\ref{disc01xyz}, the condition $\pd_R(N') \leq r$ implies that
$\pcpd_R(\Otimes C{N'})\leq r$.
Thus, with $N:=\Otimes{C}{N'}$, the sequence~\eqref{eq110530b}
satisfies the conclusion of Theorem~\ref{thm110527b}\eqref{thm110527b2}.
\qed
\end{proofn}

The remainder of the proof of Theorem~\ref{thm110527b} is similar to the above proof, 
but it requires a bit more technology.

\begin{fact} \label{disc01xy}
Let  $C$ be 
a semidualizing $R$-module.  

(a)
The Auslander class $\catac(R)$ contains
all $R$-modules of finite projective dimension,
and the Bass class $\catbc(R)$ contains
all $R$-modules of finite injective dimension; 
see~\cite[Lemmas 4.1 and 5.1, and Corollary 6.3]{holm:fear}.
Moreover, it is straightforward to show that  the Bass  class satisfies the following local-global principal:
For an $R$-module $M$, the following conditions are equivalent:
\begin{enumerate}[(i)]
\item $M\in\catbc(R)$;
\item $U^{-1}M\in\catb_{U^{-1}C}(U^{-1}R)$ for each multiplicatively closed subset $U\subseteq R$;
\item $M_{\p}\in\catb_{C_{\p}}(R_{\p})$ for each $\p\in\spec(R)$; and
\item $M_{\m}\in\catb_{C_{\m}}(R_{\m})$ for each maximal ideal $\m\in\supp_R(M)$.
\end{enumerate}
It follows that $\catbc(R)$ contains every $R$-module of locally finite injective dimension.
The Auslander class satisfies an analogous local-global principal.

(b)
The Auslander and Bass  classes also satisfy the two-of-three property
by~\cite[Corollary 6.3]{holm:fear}.
That is, given a short exact sequence $0\to M'\to M\to M''\to 0$ of $R$-module homomorphisms,
if two of the modules in the sequence are in $\catac(R)$, then so is the third module,
and similarly for $\catbc(R)$.

(c)
From~\cite[(2.8)]{takahashi:hasm} we know that an $R$-module $M$
is in $\catbc(R)$ if and only if $\Hom CM\in\catac(R)$,
and that $M\in\catac(R)$ if and only if $C\otimes_RM\in\catbc(R)$.
This is known as \emph{Foxby equivalence}.
\end{fact}

\begin{defn} \label{notation08bx}
Let $C$ be a semidualizing $R$-module.
A finitely generated $R$-module $G$ is
\emph{totally $C$-reflexive} if 
the natural  map $G\to\Hom{\Hom GC}{C}$ is an isomorphism, and
$\Ext{\geq 1}{G}{C}=0=\Ext{\geq 1}{\Hom{G}{C}}{C}$.
Let $\catgc(R)$ denote the class of totally $C$-reflexive $R$-modules,
and set $\catg(C)=\catgc(R)\cap\catbc(R)$.
In the case $C=R$ we use the more common terminology
``totally reflexive''
and the notation $\catg(R)=\catg_R(R)=\catg(R)\cap\catb_R(R)$.
We  abbreviate as follows:
\begin{align*}
\gcdim_R(-)
&=\text{the homological dimension obtained from resolutions  in $\catgc(R)$}
\\
\catg(C)\text{-}\pd_R(-)
&=\text{the homological dimension obtained from resolutions  in $\catg(C)$}
\\
\gdim_R(-)
&=\text{the homological dimension obtained from resolutions  in $\catg(R)$.}
\end{align*}
\end{defn}

The following facts are included for perspective.

\begin{fact}
Let $C$ be a semidualizing $R$-module, and let $M$ be a finitely generated $R$-module.
Because of the containments $\catpc(R)\subseteq\catg(C)\subseteq\catgc(R)$,
one has $\gcdim_R(M)\leq\catg(C)\text{-}\pd_R(M)\leq\pcpd_R(M)$ with equality to the left of any finite quantity.
In particular, the case $C=R$ says that
$\gdim_R(M)=\catg(R)\text{-}\pd_R(M)\leq\pd_R(M)$
since $\catg(R)=\catg_R(R)$, with equality holding when $\pd_R(M)<\infty$.
\end{fact}

The next  lemma explains how these homological dimensions are connected.

\begin{lem}[\protect{\cite[Lemma 2.9]{sather:tate1}}] \label{lem0701x}
Let $C$ be a semidualizing $R$-module. For a finitely generated $R$-module $M$,
the following conditions are equivalent:
\begin{enumerate}[\quad\rm(i)]
\item \label{lem0701xi}
$\opg(C)\text{-}\pd_R(M)<\infty$;
\item \label{lem0701xii}
$\gcdim_R(M)<\infty$ and $M\in\catbc(R)$; and
\item \label{lem0701xiii}
$\gdim_R(\hom CM)<\infty$ and $M\in\catbc(R)$.
\end{enumerate}
When these conditions are satisfied, we have
\begin{align*}
\opg(C)\text{-}\pd_R(M)
&=\gcdim_R(M)
=\gdim_R(\hom CM).
\end{align*}
\end{lem}

Now we are in position to complete the proof of Theorem~\ref{thm110527b}.

\begin{proofn}[Proof of Theorem~\ref{thm110527b}\eqref{thm110527b1}]
\label{proof110530b}
Assume that $M$ is a non-zero finitely generated $R$-module such that
$\gcdim_R(M) = r < \infty$ and $M\in\catbc(R)$.
Then Lemma~\ref{lem0701x} implies that
$\gdim_R(\Hom CM)=r<\infty$.
Since $M\neq 0$,  we have $\Hom CM\neq 0$, by Fact~\ref{fact110527a}(b), so
Theorem~\ref{thm110527a}\eqref{thm110527a1}
provides an $R$-regular sequence $\underline z = z_1, \ldots, z_r$,
integers $n_0, n_1, \ldots , n_r  \geq 0$ such that $n\geq 1$, and an exact sequence
\begin{equation}\label{eq110530c}
0 \to \Hom CM \to Z' \to N' \to 0
\end{equation}
with $Z' = \oplus_{i=0}^r(R/(\underline z_{[i]}))^{n_i}$ and $\gdim_R(N') \leq r$.

The condition $M\in\catbc(R)$ implies that $\Hom CM\in\catac(R)$ by Fact~\ref{disc01xy}(c),
and Fact~\ref{disc01xy}(a) implies that $Z'\in\catac(R)$.
Also from Fact~\ref{disc01xy}(a)--(b), we conclude that $N'\in\catac(R)$, which implies that $\Tor 1C{N'}=0$.
Thus, an application of $\Otimes C-$ to the sequence~\eqref{eq110530c}
yields the next exact sequence:
\begin{equation}\label{eq110530d}
0 \to \Otimes{C}{\Hom CM} \to \Otimes{C}{Z'} \to \Otimes{C}{N'} \to 0.
\end{equation}
The assumption $M\in\catbc(R)$ gives an isomorphism $\Otimes{C}{\Hom CM}\cong M$. 
The equality $Z' = \oplus_{i=0}^r(R/(\underline z_{[i]}))^{n_i}$ implies that
$Z:=\Otimes{C}{Z'} = \oplus_{i=0}^r(C/(\underline z_{[i]})C)^{n_i}$.

The condition $N'\in\catac(R)$ implies that $N:=\Otimes C{N'}\in\catbc(R)$ by Fact~\ref{disc01xy}(c),
and $N'\cong\Hom CN$ by the definition of what it means for $N'$ to be in $\catac(R)$.
Thus, because of Lemma~\ref{lem0701x}, we have 
$$\gcdim_R(N)=\gdim(\Hom CN)=\gdim_R(N') \leq r.$$
Thus,  the sequence~\eqref{eq110530d}
satisfies the conclusion of Theorem~\ref{thm110527b}\eqref{thm110527b1}.
\qed
\end{proofn}

\section{Proofs of Theorems~\ref{thm110608a} and~\ref{thm110530b}} \label{sec110608a}

We continue with a definition due to Enochs and Jenda~\cite{enochs:gipm}.
\renewcommand{\gid}{\operatorname{Gid}}

\begin{defn} \label{notation110609a}
A \emph{complete injective resolution}
is an exact complex $Y$ of injective $R$-modules 
such that
$\Hom JY$ is exact for each injective $R$-module $J$.
An $R$-module $N$ is \emph{G-injective} if there
exists a complete injective resolution $Y$ such that $N\cong\Ker(\partial^Y_0)$,
and $Y$ is a \emph{complete injective resolution of $N$}.
Let
$\catgi(R)$ denote the class of G-injective $R$-modules, and let
$\gid_R(-)$
denote the homological dimension obtained from coresolutions  in $\catgi(R)$.
We say that an $R$-module $M$ has
\emph{locally finite G-injective
dimension} provided that $\gid_{R_{\m}}(M_{\m})<\infty$ for each maximal ideal $\m\subset R$.
\end{defn}

\begin{disc}\label{disc110726a}
Using existing technology, the modules of finite G-injective dimension behave best when $R$ is Cohen-Macaulay 
with a dualizing module.\footnote{Most good behavior is also known when $R$
has a dualizing
complex, but we restrict our attention to the Cohen-Macaulay case. For instance, the conclusion of Lemma~\ref{lem110610a} holds when
$R$ is only assumed to have a pointwise dualizing complex; the proof is the same, using results from~\cite{christensen:ogpifd}
in place of the results from~\cite{enochs:fdgipm}.}
For instance, in general it is not known whether the G-injective dimension localizes;
but it is known to localize when $R$ has a dualizing module. This is due to the following connection with
Bass classes, the local case of which is from~\cite{enochs:fdgipm}.\footnote{It is worth noting that the results in~\cite{enochs:fdgipm} assume that $R$ is local,
hence with finite Krull dimension; also, the results of~\cite{christensen:ogpifd} assume implicitly that $R$ has finite Krull dimension
since the definition of a dualizing complex used there includes an assumption of finite injective dimension.
Contrast this with our definition of pointwise dualizing module, and with Grothendieck's definition of a pointwise dualizing complex from~\cite{hartshorne:rad}.}
\end{disc}

\begin{lem}\label{lem110610a}
Assume that $R$ is Cohen-Macaulay with a pointwise dualizing module.
Let $M$ be an $R$-module.
Then an $R$-module $M$ has locally finite G-injective dimension
if and only if 
$M\in\catb_{\omega}(R)$.
\end{lem}

\begin{proof}
If $R$ is local, then the result follows from~\cite[Proposition~1.4 and Theorems~1.6 and~2.5]{enochs:fdgipm}.
By definition, the condition ``$M$ has locally finite G-injective dimension'' is a local condition.
Fact~\ref{disc01xy}(b) shows that the condition ``$M\in\catb_{\omega}(R)$'' is also a local condition.
Note that Fact~\ref{fact110527a}(a) implies that for each maximal ideal $\m\subset R$, the localization
$\omega_{\m}$ is a dualizing module for the Cohen-Macaulay local ring $R_{\m}$.
Thus, the general result follows from the local case.
\end{proof}

The next result shows that the G-dimension of a finitely generated module behaves like its projective dimension.

\begin{lem}\label{lem110610b}
Let $M$ be a finitely generated $R$-module.
The following conditions are equivalent:
\begin{enumerate}[\quad\rm(i)]
\item \label{lem110610b1}
$\gdim_{R}(M)<\infty$;
\item \label{lem110610b2}
$\gdim_{R_{\m}}(M_{\m})<\infty$ for each maximal ideal $\m\subset R$.
\end{enumerate}
When $R$ is Cohen-Macaulay with a pointwise dualizing module $\omega$,
these conditions are equivalent to the following:
\begin{enumerate} 
\item[\rm(iii)] \label{lem110610b3}
$M\in\cata_{\omega}(R)$.
\end{enumerate}
\end{lem}

\begin{proof}
The implication $\eqref{lem110610b1}\implies\eqref{lem110610b2}$ follows from the inequality $\gdim_{R_{\m}}(M_{\m})\leq\gdim_R(M)$,
which is straightforward to verify.
The implication
$\eqref{lem110610b2}\implies\eqref{lem110610b1}$
is from~\cite[Theorem 3.3]{avramov:rrc1}.
When $R$ is Cohen-Macaulay with a pointwise dualizing module $\omega$,
the equivalence $\eqref{lem110610b2}\iff$(iii)
follows from the local results~\cite[Proposition~1.3 and Theorems~1.6 and~2.1]{enochs:fdgipm}
as in the proof of Lemma~\ref{lem110610a}.\footnote{Note that this uses the characterization of totally reflexive modules
in terms of ``complete resolutions'' found in~\cite[(4.4.4)]{avramov:svcci};
see also~\cite[Theorem 3.1]{avramov:aratc}. 
Specifically, an $R$-module $M$ is totally reflexive if and only if there is an exact complex $F=\cdots\xra{\partial^F_1}F_0\xra{\partial^F_0}F_{-1}\xra{\partial^F_{-1}}\cdots$
such that $\Hom FR$ is exact and $M\cong\im(\partial^F_0)$.
Note that the local result was also announced  in~\cite[Corollary 3.3]{foxby:gdcmr}.}
\end{proof}

\begin{proofn}[Proof of Theorem~\ref{thm110608a}]
\label{proof110609a}
Assume that $R$ is Cohen-Macaulay  with a pointwise dualizing
module $\omega$ and that $M$ is a non-zero finitely generated $R$-module with locally finite G-injective
dimension. 
Lemma~\ref{lem110610a} implies that $M \in \mathcal B_{\omega}(R)$.
Using Foxby equivalence (from Fact~\ref{disc01xy}(c)) we conclude that $\Hom\omega M\in\cata_{\omega}(R)$,
so Lemma~\ref{lem110610b} implies that $r:=\gdim_R(\Hom\omega M)<\infty$.
The fact that $R$ is noetherian implies that $\Hom\omega M$ is finitely generated.
(The proof  concludes as in Proof~\ref{proof110530b}. We include the details for the sake of thoroughness.)

Since $M\neq 0$,  we have $\Hom \omega M\neq 0$, by Fact~\ref{fact110527a}(b), so
Theorem~\ref{thm110527a}\eqref{thm110527a1}
implies that there are an $R$-regular sequence $\underline z = z_1, \ldots, z_r$,
integers $n_0, n_1, \ldots , n_r  \geq 0$ with $n_r\geq 1$, and an exact sequence
\begin{equation}\label{proof110609a1}
0 \to \Hom \omega M \to Z' \to N' \to 0
\end{equation}
with $Z' = \oplus_{i=0}^r(R/(\underline z_{[i]}))^{n_i}$ and $\gdim_R(N') \leq r$.
Fact~\ref{disc01xy}(a) implies that $Z'\in\cata_\omega(R)$.
From Fact~\ref{disc01xy}(b), we conclude that $N'\in\cata_\omega(R)$, which implies that $\Tor 1\omega{N'}=0$.
Thus, an application of $\Otimes \omega-$ to the sequence~\eqref{proof110609a1}
yields the next exact sequence:
\begin{equation}\label{proof110609a2}
0 \to \Otimes{\omega}{\Hom \omega M} \to \Otimes{\omega}{Z'} \to \Otimes{\omega}{N'} \to 0.
\end{equation}
The assumption $M\in\catb_\omega(R)$ gives an isomorphism $\Otimes{\omega}{\Hom \omega M}\cong M$. 
The equality $Z' = \oplus_{i=0}^r(R/(\underline z_{[i]}))^{n_i}$ implies that
$Z:=\Otimes{\omega}{Z'} = \oplus_{i=0}^r(\omega/(\underline z_{[i]})\omega)^{n_i}$.

The condition $N'\in\cata_\omega(R)$ implies that $N:=\Otimes \omega{N'}\in\catb_\omega(R)$ by Fact~\ref{disc01xy}(c),
and $N'\cong\Hom \omega N$ by the definition of what it means for $N'$ to be in $\cata_\omega(R)$.
Thus,  we have 
$\gdim(\Hom \omega N)=\gdim_R(N') \leq r$.
So,  the sequence~\eqref{proof110609a2}
satisfies the conclusion of Theorem~\ref{thm110608a}.
\qed
\end{proofn}

For our next results, we need a better understanding of the relationship between semidualizing modules and 
a pointwise dualizing module.

\begin{lem}\label{lem110611a}
Assume that $R$ is Cohen-Macaulay with a pointwise dualizing module $\omega$.
If $C$ is a semidualizing $R$-module, then so is $\Hom C\omega$.
\end{lem}

\begin{proof}
In the local case, this result is standard; see, e.g., \cite[Facts 1.18--1.19]{sather:bnsc}.
Since the semidualizing and (pointwise) dualizing conditions are local by Fact~\ref{fact110527a}(a),
the general case of the result follows.
\end{proof}

The next two lemmas elaborate on the local-global behavior for our invariants.

\begin{lem}\label{lem110611d}
Let $M$ be a finitely generated $R$-module, and
let $C$ be a semidualizing $R$-module.
Then $\pcpd_{R}(M)<\infty$ if and only if
$\text{P}_{C_{\m}}\text{-}\pd_{R_{\m}}(M_{\m})<\infty$ for each maximal ideal $\m\subset R$.
\end{lem}

\begin{proof}
The forward implication  follows from the inequality $\text{P}_{C_{\m}}\text{-}\dim_{R_{\m}}(M_{\m})\leq\pcpd_R(M)$,
which is straightforward to verify.
For the converse, assume that $\text{P}_{C_{\m}}\text{-}\pd_{R_{\m}}(M_{\m})<\infty$ for each maximal ideal $\m\subset R$.
It follows that the module
$\Hom CM_{\m}\cong\Hom[R_{\m}]{C_{\m}}{M_{\m}}$
has finite projective dimension over $R_{\m}$ for all $\m$. Hence, the finitely generated $R$-module $\Hom CM$ has finite projective dimension,
so Fact~\ref{disc01xyz} implies that $\pcpd_{R}(M)$ is finite.
\end{proof}

The next lemma  is a souped-up version of a special case of a result of Takahashi and White that is documented in~\cite{sather:tate1}.

\begin{lem} \label{lem110612a}
Assume that $R$ is Cohen-Macaulay with a pointwise dualizing module $\omega$.
For each finitely generated $R$-module $M$, one has
$\catp_{\omega}\text{-}\pd_R(M)<\infty$
if and only if $M$ has locally finite injective dimension.
\end{lem}

\begin{proof}
When $\omega$ is a dualizing module for $R$, i.e., when $R$ has finite Krull dimension, this result is from~\cite[Lemma 2.7]{sather:tate1}.
In particular, this takes care of the local case.
The general case follows from the local case, by Lemma~\ref{lem110611d}.
\end{proof}

Here are some more notions from~\cite{holm:smarghd}.

\begin{defn} \label{notation08b}
Let $C$ be a semidualizing $R$-module, and let
$\catic(R)$ denote the class of modules $N\cong \Hom CI$ with $I$ injective.
Modules in  $\catic(R)$ are called \emph{$C$-injective}.
We let
$\icid_R(-)$
denote the homological dimension obtained from coresolutions  in $\catic(R)$.
We say that an $R$-module $M$ has
\emph{locally finite $\icid$} provided that 
$\text{I}_{C_{\m}}\text{-}\id_{R_{\m}}(M_{\m})<\infty$ for each maximal ideal $\m\subset R$.
\end{defn}

\begin{fact} \label{fact110612a}
Let $C$ be  a
semidualizing $R$-module.  
The Auslander class $\catac(R)$ contains every  $R$-module
of finite $\cati_C$-injective dimension;
see~\cite[Lemmas 4.1 and 5.1]{holm:fear}.
Given an $R$-module $M$, one has
$\icid_R(M)=\id_R(\Otimes CM)$
and 
$\id_R(M)=\icid_R(\Hom CM)$
by~\cite[Theorem 2.11]{takahashi:hasm}.
In particular, one has
$\icid_R(M)<\infty$ if and only if
$\id_R(\Otimes CM)<\infty$,
and one has
$\id_R(M)<\infty$ if and only if
$\icid_R(\Hom CM)<\infty$.
It follows that $M$ has locally finite $\icid$
if and only if $\Otimes CM$ has locally finite injective dimension,
and 
$M$ has locally finite injective dimension
if and only if $\Hom CM$ has locally finite $\icid$.
\end{fact}

\begin{lem}\label{lem110612b}
Assume that $R$ is Cohen-Macaulay with a pointwise dualizing module $\omega$.
Let $C$ be a semidualizing $R$-module, and let $x=x_1,\ldots,x_i\in R$ be an $R$-regular sequence.
Then $\Hom C{\omega/(\underline x)\omega}\cong\cd/(\underline x)\cd$ where $\cd=\Hom C\omega$.
\end{lem}

\begin{proof}
Let $K$ be the Koszul complex $K^R(\underline x)$, and let $K^+$ denote the augmented Koszul complex
$$K^+=(0\to R\to\cdots\to R\to \underbrace{R/(\underline x)}_{\text{deg. $-1$}}\to 0).$$
Since $x$ is $R$-regular, $K^+$ is an exact sequence of $R$-modules of finite projective dimension.
Thus, each module $K^+_j$ is in $\cata_{\omega}(R)$, so the induced complex
$$\Otimes{\omega}{K^+}=(0\to \omega\to\cdots\to \omega\to \underbrace{\omega/(\underline x)\omega}_{\text{deg. $-1$}}\to 0)$$
is an exact sequence of $R$-modules locally of finite injective dimension.
In particular, the modules in this exact sequence are in $\catbc(R)$,
so the next sequence is also exact:
\begin{align*}
\Hom{C}{\Otimes{\omega}{K^+}}\hspace{-2cm}\\
&=(0\to \Hom{C}{\omega}\to\cdots\to \Hom{C}{\omega}\to \underbrace{\Hom{C}{\omega/(\underline x)\omega}}_{\text{deg. $-1$}}\to 0).
\end{align*}
It is straightforward to show that the differential on this sequence in positive degrees is the same as the differential on
$\Otimes {\cd}{K}$.
It follows that 
$$\Hom{C}{\omega/(\underline x)\omega}
\cong\HH_0(\Otimes {\cd}{K})
\cong\cd/(\underline x)\cd$$
as desired.
\end{proof}

\begin{proofn}[Proof of Theorem~\ref{thm110530b}\eqref{thm110530b2}]
\label{proof110609b}
Assume that $R$ is Cohen-Macaulay  with a pointwise dualizing
module $\omega$ and that $M$ is a non-zero finitely generated $R$-module with locally finite $\icid$. 
Fact~\ref{fact110612a} implies that
$\Otimes CM$ has locally finite injective dimension.
Since $M\neq 0$,  we have $\Otimes C M\neq 0$, by Fact~\ref{fact110527a}(b), so
Theorem~\ref{thm110530a} provides
an integer $r = \pd_R(\Hom\omega {\Otimes CM})$, a proper
$R$-regular sequence $\underline z = z_1, \ldots , z_r$, non-negative integers $n_0, n_1, \ldots , n_r$ with $n_r\geq 1$, and an
exact sequence
\begin{equation}\label{proof110609b1}
0 \to \Otimes CM \to Z' \to N' \to 0
\end{equation}
with $Z' = \oplus_{i=0}^r(\omega/(\underline z_{[i]})\omega)^{n_i}$, where $N'$
has locally finite injective dimension and $\pd_R(\Hom\omega {N'})\leq r$. 

As $\Otimes CM$  has locally finite injective dimension, 
we have $\Otimes CM\in\catbc(R)$ by Fact~\ref{disc01xy}(b), so $\Ext 1C{\Otimes CM}=0$ and $\Hom C{\Otimes CM}\cong M$.
Thus, an application of $\Hom C-$ to the sequence~\eqref{proof110609b1}
yields the next exact sequence:
\begin{equation}\label{proof110609b2}
0 \to M \to \Hom C{Z'} \to \Hom C{N'} \to 0.
\end{equation}
Since $N'$ has locally finite injective dimension, Fact~\ref{fact110612a}
implies that the $R$-module $N:=\Hom C{N'}$ has locally finite $\icid$.
With $Z:=\Hom C{Z'}$,
it remains to show that 
$Z \cong \oplus_{i=0}^r(\cd/(\underline z_{[i]})\cd)^{n_i}$  and $\pd_R(\Hom\omega{\Otimes CN})\leq r$.
The first of these follows from the next sequence of isomorphisms
\begin{align*}
Z
&\cong\Hom C{\oplus_{i=0}^r(\omega/(\underline z_{[i]})\omega)^{n_i}}\\
&\cong\oplus_{i=0}^r\Hom C{\Otimes{\omega}{(R/(\underline z_{[i]})}}^{n_i}\\
&\cong\oplus_{i=0}^r(\cd/(\underline z_{[i]})\cd)^{n_i}
\end{align*}
where the last isomorphism is from Lemma~\ref{lem110612b}.
To complete the proof, observe that since $N'$ has locally finite injective dimension,
it is in $\catbc(R)$, so we have $N'\cong \Otimes CN$.
This implies that
$\pd_R(\Hom\omega{\Otimes CN})=\pd_R(\Hom\omega{N'})
\leq r$,
as desired.
\qed
\end{proofn}

In the next definition, the special case $C=R$ recovers the
complete injective resolutions and Gorenstein injective modules.

\begin{defn} \label{notation08b'}
Let $C$ be a semidualizing $R$-module.
A \emph{complete $\catic\cati$-coresolu-tion}
is a complex $Y$ of $R$-modules 
such that
$Y$ is exact and $\Hom UY$ is exact for each $U\in\catic(R)$, 
$Y_i$ is injective for  $i\leq 0$, and $Y_i$ is  $C$-injective for $i>0$.
An $R$-module $N$ is \emph{$\text{G}_C$-injective} if there
exists a complete $\catic\cati$-coresolution $Y$ such that $N\cong\Ker(\partial^Y_0)$,
and $Y$ is a \emph{complete $\catic\cati$-coresolution of $N$}.
Let
$\catgic(R)$ denote the class of $\text{G}_C$-injective $R$-modules,
and set $\catg(\catic)=\catgic(R)\cap\catac(R)$.
We  let $\catgic\text{-}\id_R(-)$ and $\opg(\catic)\text{-}\id_R(-)$ denote
the homological dimensions obtained from coresolutions  in $\catgic(R)$
and $\opg(\catic)$, respectively.

An $R$-module $M$ has
\emph{locally finite $\opg(\catic)\text{-}\id$} provided that $\opg(\cati_{C_{\m}})\text{-}\id_{R_{\m}}(M_{\m})<\infty$ for each maximal ideal $\m\subset R$.
An $R$-module $M$ has
\emph{locally finite $\opg\catic\text{-}\id$} provided that $\opg\cati_{C_{\m}}\text{-}\id_{R_{\m}}(M_{\m})<\infty$ for each maximal ideal $\m\subset R$.
\end{defn}

The next  lemma explains the relation between the different Gorenstein injective dimensions.

\begin{lem}[\protect{\cite[Lemma 2.10]{sather:tate1}}] \label{lem0701'}
Let $C$ be a semidualizing $R$-module. For an $R$-module $M$,
the following conditions are equivalent:
\begin{enumerate}[\quad\rm(i)]
\item \label{lem0701'i}
$\opg(\catic)\text{-}\id_R(M)<\infty$;
\item \label{lem0701'ii}
$\opg\catic\text{-}\id_R(M)<\infty$ and $M\in\catac(R)$; and
\item \label{lem0701'iii}
$\gid_R(C\otimes_RM)<\infty$ and $M\in\catac(R)$.
\end{enumerate}
When these conditions are satisfied, we have
\begin{align*} \opg(\catic)\text{-}\id_R(M)
  &
  =\opg\catic\text{-}\id_R(M)
  =\gid_R(C\otimes_RM).  
\end{align*}
\end{lem}

We actually need the following corollary of the previous result.

\begin{lem} \label{lem110613a}
Let $C$ be a semidualizing $R$-module. For an $R$-module $M$,
the following conditions are equivalent:
\begin{enumerate}[\quad\rm(i)]
\item \label{lem110613ai}
$M$ has locally finite $\opg(\catic)\text{-}\id$;
\item \label{lem110613aii}
$M$ has locally finite $\opg\catic\text{-}\id$ and $M\in\catac(R)$; and
\item \label{lem110613aiii}
$\Otimes CM$ has locally finite $\gid$ and $M\in\catac(R)$.
\end{enumerate}
\end{lem}

\begin{proof}
This follows from the definitions of the locally finite 
Gorenstein injective dimensions
and the local-global principal for Auslander classes from Fact~\ref{disc01xy}(b).
\end{proof}

\begin{proofn}[Proof of Theorem~\ref{thm110530b}\eqref{thm110530b1}]
\label{proof110609c}
Assume that $R$ is Cohen-Macaulay  with a pointwise dualizing
module $\omega$ and that $M$ is a non-zero finitely generated $R$-module in $\catac(R)$ with locally finite $\opg\icid$. 
Lemma~\ref{lem110613a} implies that 
$\Otimes CM$ has locally finite G-injective dimension.
Since $M\neq 0$,  we have $\Otimes C M\neq 0$, by Fact~\ref{fact110527a}(b), so
Theorem~\ref{thm110608a} provides
an integer $r = \gdim_R(\Hom\omega {\Otimes CM})$, a proper
$R$-regular sequence $\underline z = z_1, \ldots , z_r$, non-negative integers $n_0, n_1, \ldots , n_r$ with $n_r\geq 1$, and an
exact sequence
\begin{equation}\label{proof110609c1}
0 \to \Otimes CM \to Z' \to N' \to 0
\end{equation}
such that $Z' = \oplus_{i=0}^r(\omega/(\underline z_{[i]})\omega)^{n_i}$, $N'\in\catb_{\omega}(R)$, and $\gdim_R(\Hom\omega {N'})\leq r$. 

Since $M$ is in $\catac(R)$,
we have $\Ext 1C{\Otimes CM}=0$ and $\Hom C{\Otimes CM}\cong M$.
Thus, an application of $\Hom C-$ to the sequence~\eqref{proof110609c1}
yields the next exact sequence:
\begin{equation}\label{proof110609c2}
0 \to M \to \Hom C{Z'} \to \Hom C{N'} \to 0.
\end{equation}
Note that $Z'$ has locally finite injective dimension, so we have $Z'\in\catbc(R)$, which implies that $\Hom{C}{Z'}$
and $\Hom{C}{N'}$ are in $\catac(R)$
by Fact~\ref{disc01xy}(b)--(c).
With $Z:=\Hom C{Z'}$, we have $\gdim_R(\Hom\omega{\Otimes CN})\leq r$  and $Z \cong \oplus_{i=0}^r(\cd/(\underline z_{[i]})\cd)^{n_i}$, as in Proof~\ref{proof110609b}.
\qed
\end{proofn}

\section*{Acknowledgments}

We are grateful to the referee for her/his thoughtful comments.

\providecommand{\bysame}{\leavevmode\hbox to3em{\hrulefill}\thinspace}
\providecommand{\MR}{\relax\ifhmode\unskip\space\fi MR }
\providecommand{\MRhref}[2]{%
  \href{http://www.ams.org/mathscinet-getitem?mr=#1}{#2}
}
\providecommand{\href}[2]{#2}

\end{document}